\documentclass[11pt,a4paper,oneside]{amsart}

\usepackage{graphicx}
\usepackage{amscd}
\usepackage{amsmath}
\usepackage{amssymb}

\newcommand{\bR}{\mathbb{R}}

\newcommand{\bZ}{\mathbb{Z}}

\usepackage[a4paper,left=0.8in,right=0.8in,top=1in,bottom=1in]{geometry}

\theoremstyle{plain}
\newtheorem{theorem}{Theorem}[section]

\newtheorem{lemma}{Lemma}[section]
\newtheorem{proposition}{Proposition}[section]

\theoremstyle{definition}

\theoremstyle{remark}
\newtheorem{remark}{Remark}[section]

\numberwithin{equation}{section}

\title[Lacunary elliptic maximal]{Lacunary elliptic maximal operator on the Heisenberg group}

\author[Joonil Kim]{Joonil Kim}
\address[Joonil Kim]{Department of Mathematics, Yonsei University, Seoul 120-749, Republic of Korea}
\email{jikim7030@yonsei.ac.kr}

\author[Jeongtae Oh]{Jeongtae Oh}
\address[Jeongtae Oh]{Research Institute of Mathematics, Seoul National University, Seoul 08826, Republic of Korea}
\email{ojt0117@snu.ac.kr}

\begin{document}

\begin{abstract}
In this paper, we prove \( L^p \) boundedness results for lacunary elliptic maximal operators on the Heisenberg group. Furthermore, we extend these \( L^p \) estimates from skew-symmetric matrices, which naturally arise in Heisenberg group operations, to arbitrary matrices \( A \), investigating how the curvature induced by \( A \) governs the \( L^p \) boundedness of lacunary circular and elliptic maximal operators. Specifically, we provide necessary and sufficient conditions on \( A \) that determine whether these operators are bounded or unbounded on \( L^p \).
\end{abstract}

\maketitle

\setcounter{tocdepth}{1}
\tableofcontents

\newtheorem*{zyc}{Zygmund Conjecture}
\newtheorem*{opb}{Open Problem}

\section{Introduction and statement of results}

In recent years, there has been a growing interest in establishing \( L^p \) estimates for lacunary dilated geometric maximal functions on the Heisenberg groups. Notably, Bagchi, Hait, Roncal, and Thangavelu \cite{MR4250270} obtained  \( L^p \) estimates for lacunary spherical maximal functions on the Heisenberg group \( \mathbb{H}^n \) for \( n \geq 2 \), utilizing the group Fourier transform in conjunction with spectral analysis. In this paper, we focus on the case \( n = 1 \) and prove \( L^p \) estimates for lacunary elliptic maximal functions on the Heisenberg group. We employ the group Fourier transform as in \cite{MR4250270, MR4170795} and additionally utilize Littlewood-Paley theory on the Heisenberg groups along with decay estimates for oscillatory integral operators, following \cite{MR1924572}. Furthermore, we extend these \( L^p \) estimates from skew-symmetric matrices, which naturally arise in Heisenberg group operations, to arbitrary matrices \( A \), investigating how the curvature induced by \( A \) governs the \( L^p \) boundedness of lacunary circular and elliptic maximal operators. 

To introduce our results, let $\mathbb{M}_{2}(\mathbb{R})$ be the space of $2\times 2$ real matrices. Let $d\sigma$ denotes the canonical Lebesgue measure on $S^{1}$. Consider the averages over ellipses lying in $\mathbb{R}^2$, acting on function defined on $\mathbb{R}^2\times \mathbb{R}$ as follows.
\begin{align*}
	{E}^{A}_{t_1,t_2}f(x,x_3):&=\int_{ S^{1}}f(x-(t_1y_1,t_2y_2),x_3-x^tA\begin{pmatrix}t_1y_1\\t_2y_2\end{pmatrix}) d\sigma(y),
	\end{align*}
for $(x,x_3)\in \mathbb{R}^2\times\mathbb{R}$ and $t_1,t_2>0$. It is well known that we can identify the Heisenberg group $\mathbb{H}^1$ with $\mathbb{R}^{3}$ via the group law given by $(x,x_{d+1})\cdot_J (y,y_{d+1})=(x+y,x_{3}+y_{3}+x^tJy)$ where $J=\Bigg(\begin{matrix}
0&-2\\2&0	
\end{matrix}\Bigg)
$. 
Define the measure $d\sigma_1$ as the canonical Lebesgue measure supported on $S^{1}\times \{ 0 \} \subset \mathbb{H}^1$, with $\langle d\sigma_1^{t,s},f \rangle=\int f
((ty_1,sy_2,0))d\sigma(y)$. Then it can be viewed ${E}^{J}_{s,t}f(x,x_3)=f\ast_J \sigma_1^{t,s}$ where the convolution $\ast_J$ is defined on the Heisenberg group. For $f\in \mathcal{S}(\mathbb{R}^3)$, define the lacunary circular maximal operator $\mathcal{E}_{A}^1$ and the lacunary elliptic maximal operator $\mathcal{E}_{A}^2$ as
\begin{align*}
 	\mathcal{E}_{A}^1f(x,x_{3}):&=\sup_{k\in \bZ}|{E}^{A}_{2^k,2^k}f(x,x_{3})|,\\
	\mathcal{E}_{A}^2f(x,x_{3}):&=\sup_{k_1, k_2\in \bZ}|{E}^{A}_{2^{k_1},2^{k_2}}f(x,x_{3})|.
\end{align*}

We begin by presenting an example in Theorem \ref{3mt0}, where we demonstrate that the operators $\mathcal{E}_{A}^1$ and $\mathcal{E}_{A}^2$ are unbounded on $L^p$ spaces for certain matrices $A$. This example illustrates the failure of $L^p$ boundedness in the general case. In contrast, in Theorem \ref{3mt1}, we prove the $L^p$ boundedness of lacunary elliptic maximal operators on the Heisenberg group.
\begin{theorem}\label{3mt0}
For $A \in \mathbb{M}_{2}(\mathbb{R})$, if $A = cI$ for some $c \in \mathbb{R}\setminus \{0\}$, then the operator $\mathcal{E}_{A}^1$ is unbounded on $L^p(\mathbb{R}^3)$ for $0<p<\infty$. Furthermore, if $A = \begin{pmatrix} c & 0 \\ 0 & c \cdot 2^{2a} \end{pmatrix}$ for $a \in \mathbb{Z}$ and $c \in \mathbb{R}\setminus \{0\}$, then the operator $\mathcal{E}_{A}^2$ is unbounded on $L^p(\mathbb{R}^3)$ for $0<p<\infty$.
\end{theorem}

 \begin{theorem}\label{3mt1}
Let $J$ be a skew--symmetric matrix. Then $\mathcal{E}_{J}^2$ is bounded on $L^p(\mathbb{H}^1)$ for $1<p\leq \infty$.
\end{theorem}
From Theorems \ref{3mt0} and Theorem \ref{3mt1}, it is natural to ask for the necessary and sufficient conditions on the matrix $ A $ for the operators $\mathcal{E}_{A}^1 $, $\mathcal{E}_{A}^2 $  to be bounded on $ L^p $.
 \begin{theorem}\label{3mt2}
Let $A\in \mathbb{M}_{2}(\mathbb{R})$ and $1<p<\infty$. 
The operators $\mathcal{E}_{A}^1$ are bounded on $L^p(\mathbb{R}^3)$ if and only if $A$ is not of the form $ \begin{pmatrix} c & 0 \\ 0 & c \end{pmatrix}$  for some $c\in \mathbb{R}\setminus \{0\}$. Moreover, the operators $\mathcal{E}_{A}^2$ are bounded on $L^p(\mathbb{R}^3)$ if and only if $A$ is not of the form $ \begin{pmatrix} c & 0 \\ 0 & c \cdot 2^{2a} \end{pmatrix}$  for some $c\in \mathbb{R}\setminus \{0\}$ and $a\in\mathbb{Z}$.
\end{theorem}

\begin{remark}
An interesting point is that the $L^p$ boundedness of $\mathcal{E}_{A}^1$ and $\mathcal{E}_A^2$ can be determined by the matrix $A$ given that the $L^p$ unboundedness is rarely observed for lacunary maximal operators.
\end{remark}

\subsection*{Historical remark}
In 1976, Stein studied the spherical maximal function, proving that the operator defined by
$
Mf(x) := \sup_{r>0} \left| f \ast d\sigma_{d-1}^r \right|
$
is bounded on $ L^p(\mathbb{R}^d) $ for $ p > \frac{d}{d-1} $ when $ d \geq 3 $. Here, $ d\sigma_{d-1}^r $ denotes the normalized surface measure on the sphere $ rS^{d-1} = \{ x \in \mathbb{R}^d : |x| = r \} $. The case $ d = 2 $ was later resolved by Bourgain in 1986 \cite{MR0874045}. Around the same time, Calder\'on investigated lacunary maximal functions \cite{MR537803}, defined by $
M_Lf(x) := \sup_{k \in \mathbb{Z}} | f \ast d\sigma_{d-1}^{2^k} |,$
proving that this operator is bounded on $ L^p(\mathbb{R}^d) $ for $ 1 < p \leq \infty $. These results naturally led to questions regarding spherical averages on the Heisenberg group $ \mathbb{H}^n $. In 1997, Nevo and Thangavelu \cite{MR1448717} studied the spherical means $ f \ast_J d\sigma_{2n-1}^r $ on $ \mathbb{H}^n $, obtaining $ L^p $ estimates for $ \sup_{r > 0} | f \ast_J d\sigma_{2n-1}^r | $. Optimal ranges for $ p $ were independently determined by M\"uller and Seeger \cite{MR2063040}, and Narayanan and Thangavelu \cite{MR2121541}. However, the corresponding estimate on $ \mathbb{H}^1 $ remains open.
 For results restricted to Heisenberg radial functions, see \cite{MR4453958} and \cite{MR4566682}. 
 
 On the other hand, M\"uller and Seeger \cite{MR2063040} studied spherical maximal functions in the setting of M\'etrivier groups, which can be seen as tilted spherical maximal functions on the Heisenberg group. This type of maximal function has been extensively studied in \cite{MR4361901}, \cite{MR4523247}.  There also has been considerable research extending the study to two-step nilpotent groups. Refer to \cite{kim2020annulusmaximalaveragesvariable}, \cite{MR4581163} and \cite{MR4774069}. 
 
 The study of lacunary spherical maximal functions on $ \mathbb{H}^n $ for $ n \geq 2 $ was addressed by Bagchi, Hait, Roncal, and Thangavelu in \cite{MR4250270}, while the case $ n = 1 $ was covered by Roos, Seeger and Srivastava \cite{MR4523247} using $ L^p $-improving estimates for spherical averages. More recently, Sheri, Hickman, and Wright \cite{MR4669588} obtained $ L^p $ estimates for lacunary maximal functions on general homogeneous groups under appropriate curvature conditions. 
\subsection*{This paper} 
Our approach, similar to that of \cite{MR4250270, MR4170795}, employs the group Fourier transform on the Heisenberg group. However, our proof differs as we do not utilize spectral analysis but instead directly compute oscillatory integral operators. Moreover, we handle multiparameter lacunary maximal operators by applying the Littlewood-Paley projection operators corresponding to the Heisenberg group, based on \cite{MR1924572}.\\
\text{ } In Section \ref{gftlp}, we explain the group Fourier transform on the Heisenberg group and the various properties needed for the proof. In Section \ref{3strpf}, we outline the structure of the proof for Theorem \ref{3mt1}. We then prove Theorem \ref{3mt1} in Sections \ref{dcysection} and \ref{bstsection}, and finally, we prove Theorem \ref{3mt0} in Section \ref{Asect} and Theorem \ref{3mt2} in Section \ref{GAsec}.

\subsection*{Notation} Let $\psi:\bR\rightarrow \bR$ be a non-negative $C^{\infty}$ function supported on $[-2,2]$ such that $\psi\equiv 1$ on $[-1,1]$. Define $\varphi(t)=\psi(t)-\psi(2t)$. Also, define $\psi^c(t)=1-\psi(t)$. Note that
	$\sum_{\l\in \bZ} \varphi\left(\frac{t}{2^l}\right)=1\ \text{for }t\neq0$ and $\text{supp}(\varphi)\subset \left\{\frac{1}{2}\leq|t|\leq2\right\}$.
We shall use the notation $A\lesssim B$ when $A\leq CB$ with a constant $C>0$. Moreover, we write $A\approx B$, if $A\lesssim B$ and $B\lesssim A$. 
We denotes the convolution of $f$ and $g$ by $f\ast_J g$ on the Heisenberg group and $f\ast g$ on Euclidean space, which mean $		f\ast_J g(x):=\int f(x\cdot_Jy^{-1})g(y)dy$ and $		f\ast g(x):=\int f(x-y)g(y)dy$, respectively.

\section{Group Fourier transform and Littlewood-Paley Theorem on the Heisenberg group}\label{gftlp}
We define the group Fourier transform of $f\in L^1(\mathbb{H}^n)\bigcap L^2(\mathbb{H}^n)$ as an operator--valued mapping from $\mathbb{R}^1$ to the space of bounded operators on $L^2(\mathbb{R}^n)$ such that $\lambda\in \mathbb{R}^1  \mapsto [\hat{f}(\lambda)\phi](x)$,
\begin{align*}
[\hat{f}(\lambda)\phi](x)=\int_{\bR^1}\int_{\mathbb{R}^{2n}} e^{-2\pi i \lambda (q\xi-\frac{pq}{2}+\frac{s}{4})}\phi(\xi-p)f(p,q,s)dpdqds.
\end{align*} 
Making the change of variables gives
\begin{align}\label{gf}
[\hat{f}(\lambda)\phi](x)=\int_{\mathbb{R}^n}{\mathcal{F}}^{2,3}f\left(x-y,\frac{\lambda(x+y)}{2},\frac{\lambda}{4}\right)	\phi(y)dy,
\end{align}
where $\mathcal{F}^{2,3}$ is the Euclidean Fourier transform with respect to the second and third component of $f$.
Note that 
\begin{align}\label{sep}
	[\widehat{f\ast_J g}(\lambda)\phi](x)=\left[\hat{f}(\lambda)[\hat{g}(\lambda)\phi]\right](x)
\end{align}
for $f,g\in L^1(\mathbb{H}^n)$.
The operator $\hat{f}(\lambda)$ is actually a Hilbert--Schmidt operator and the following Plancherel theorem holds.
\begin{align*}
	\|f\|_{L^2(\mathbb{H}^n)}=\frac{1}{4}\int_{\mathbb{R}^1}\|\hat{f}(\lambda)\|^2_{HS}|\lambda|^nd\lambda,
\end{align*}
where $\|\cdot\|_{HS}$ is a Hilbert--Schmidt norm. The important thing is that the $L^2$ boundedness of the convolution operators in $\mathbb{H}^1$ can be estimated by the following proposition.
\begin{proposition}\label{H2}
	Let $G$ be a operator defined by $Gf=k\ast_J f$ on $L^2(\mathbb{H}^1).$
	Then 
	\begin{align*}
		\widehat{Gf}(\lambda)\phi=\hat{k}(\lambda)[\hat{f}(\lambda)\phi],
	\end{align*}
for all $\phi\in L^2(\mathbb{R}^1)$. Moreover, the operator norm of $G$ is given by
\begin{align*}
	\|G\|_{L^2(\mathbb{H}^1)\mapsto L^2(\mathbb{H}^1)}=\frac{1}{2}\|\hat{k}(\lambda)\|_{L^2(\mathbb{R}^1)\mapsto L^2(\mathbb{R}^1)}.
\end{align*}
\end{proposition}
See \cite{MR1744778} and \cite{MR1232192} for details of the proof.\\
For $j\in \bZ$, let us define 
\begin{align}\label{fdec}
\begin{split}
	L_j^1(y)=\mathcal{F}^{-1}[\varphi(2^j\cdot)](y_1)\delta(y_2)\delta(y_3),\\
	L_j^2(y)=\mathcal{F}^{-1}[\varphi(2^j\cdot)](y_2)\delta(y_1)\delta(y_3).
\end{split}
\end{align}
where $\delta$ is a dirac measure on $\mathbb{R}^2$. 
Set 
\begin{align*}
\mathcal{L}_{j}^\nu f:=L_{j}^{\nu}\ast_J f
\end{align*}
for $\nu=1,2$ and $\mathcal{L}_{k}^{\nu,\rm{loc}}f=\sum_{k}^{\infty}L_{j}^{\nu}\ast_{J} f$, $ \mathcal{L}_{k}^{\nu,\rm{glo}}f=f-\mathcal{L}_{k}^{\nu,\rm{loc}}f.$

In the setting of the Heisenberg group, an analogue version of the Littlewood-Paley theorem is established as follows:

\begin{lemma}\label{lpt}
For $1<p<\infty$ and $\nu=1,2$, there exists a constant $C_p$ such that 
\begin{align*}
&\left\|\left(\sum_{k\in \mathbb{Z}}|\mathcal{L}_k^{\nu}f|^2\right)^{\frac{1}{2}}\right\|_{L^p(\mathbb{H}^1)}\leq C_p \|f\|_{L^{p}(\mathbb{H}^1)}.
\end{align*}

\end{lemma}
The proof of this lemma can be found in \cite{MR1924572}.	
To further study harmonic analysis on the Heisenberg group, refer to \cite{MR1633042} and Chapter $12$ of \cite{MR1232192}.

\section{Structure of the proof for Theorem \ref{3mt1}}\label{3strpf}
Let $K=(k_1,k_2)\in \mathbb{Z}^2$ and denote $\theta_K=(2^{k_1}\cos \theta,2^{k_2}\sin \theta).$ To prove Theorem \ref{3mt1}, we shall estimate the $L^p$ boundedness of the operator $$\mathcal{E}^{J}_{2^{k_1},2^{k_2}}f(x,x_3):=\sup_{K\in \bZ^2}\int_{[0,2\pi]}|f((x,x_3)\cdot_{J}({\theta}_K,0)^{-1})|d\theta.$$ We decompose the integral $\int_{[0,2\pi]}d\theta$ into four intervals: $\int_{[{\frac{\pi}{4}},{\frac{3\pi}{4}}]}$, $\int_{\frac{\pi}{2}+[{\frac{\pi}{4}},{\frac{3\pi}{4}}]}$, $\int_{{\pi}+[{\frac{\pi}{4}},{\frac{3\pi}{4}}]}$, and $\int_{\frac{3\pi}{2}+[{\frac{\pi}{4}},{\frac{3\pi}{4}}]}$. Using an appropriate change of variables $\theta\mapsto\theta+\pi/2$, we observe that the operator norms over each interval are identical. Therefore,  to complete the argument, it is sufficient to perform the analysis on just $[\frac{\pi}{4}, \frac{3\pi}{4}]$ as in
\begin{align*}
	\sup_{K\in \bZ^2}\int_{[\frac{\pi}{4}, \frac{3\pi}{4}]}|f((x,x_3)\cdot_{J}({\theta}_K,0)^{-1})|d\theta.
\end{align*}
This expression can be written as $\sup_K|\zeta_K*_Jf(x,x_3)|$, where the measure $\zeta_K$ is defined as
\begin{align}\label{022m}
\langle \zeta_K,f \rangle= \int_{[\frac{\pi}{4}, \frac{3\pi}{4}]} f\left(2^{k_1}\cos\theta, 2^{k_2}\sin\theta,0\right) d\theta\ \text{for $f\in \mathcal{S}(\mathbb{R}^3)$}.
\end{align}
Recall that $\mathcal{L}_{k}^{\nu,\rm{loc}}f=\sum_{k}^{\infty}L_{j}^{\nu}\ast_{J} f$, $ \mathcal{L}_{k}^{\nu,\rm{glo}}f=f-\mathcal{L}_{k}^{\nu,\rm{loc}}f$ and $\mathcal{L}_{j}^\nu f:=L_{j}^{\nu}\ast_J f.$
To prove our main results, we decompose 
\begin{align}\label{md}
	\zeta_{K}\ast_J f=\zeta_{K}\ast_J\mathcal{L}_{k_1}^{1,\rm{loc}}f+\zeta_{K}\ast_J\mathcal{L}_{k_2}^{2,\rm{loc}}\mathcal{L}_{k_1}^{1,\rm{glo}}f+\sum_{\ell_1,\ell_2=0}^{\infty}\zeta_{K}\ast_J \mathcal{L}_{k_2-\ell_2}^2\mathcal{L}_{k_1-\ell_1}^{1}f.
\end{align}
To handle the first and second terms, we will prove Proposition \ref{loc}. In the proof, we control these terms through the composition of two types of maximal operators.\begin{proposition}\label{loc}
For each \( 1 < p \leq \infty \), there exists a constant \( C_p \) such that
\[
\left\| \sup_{K} \left| \zeta_{K} \ast_J \mathcal{L}_{k_1}^{1,\mathrm{loc}} f \right| \right\|_{L^p(\mathbb{H}^1)} + \left\| \sup_{K} \left| \zeta_{K} \ast_J \mathcal{L}_{k_2}^{2,\mathrm{loc}} \mathcal{L}_{k_1}^{1,\mathrm{glo}} f \right| \right\|_{L^p(\mathbb{H}^1)} \leq C_p \| f \|_{L^p(\mathbb{H}^1)}.
\]
\end{proposition}

To handle the third term of \eqref{md}, we shall show the following estimate.
 \begin{proposition}\label{3dcy}
For each \(1 < p < \infty\), there exists a constant \(C_p > 0\) such that
\begin{align}\label{4.2}
	\bigg\|\bigg(\sum_{k_1,k_2} \big|\zeta_{K} \ast_J \mathcal{L}_{k_2 - \ell_2}^{2} \mathcal{L}_{k_1 - \ell_1}^{1} f \big|^2 \bigg)^{\frac{1}{2}} \bigg\|_{L^p(\mathbb{H}^1)} \leq C_p \, 2^{-\varepsilon(\ell_1 + \ell_2)} \bigg\|\bigg(\sum_{k_1,k_2} \big|\mathcal{L}_{k_2 - \ell_2}^{2} \mathcal{L}_{k_1 - \ell_1}^{1} f \big|^2 \bigg)^{\frac{1}{2}} \bigg\|_{L^p(\mathbb{H}^1)}.
\end{align}
\end{proposition}
The estimate \eqref{4.2} follows from the following two estimates.
\begin{align}\label{4.4}
\begin{split}
	\bigg\|\bigg(\sum_{k_1,k_2}|\zeta_{K}\ast_J \mathcal{L}_{k_2-\ell_2}^{2}\mathcal{L}_{k_1-\ell_1}^1&f
|^2\bigg)^{\frac{1}{2}}\bigg\|_{L^2(\mathbb{H}^1)}\lesssim 2^{-\varepsilon(\ell_1+\ell_2)}\bigg\|\bigg(\sum_{k_1,k_2}|\mathcal{L}_{k_2-\ell_2}^{2}\mathcal{L}_{k_1-\ell_1}^1f
|^2\bigg)^{\frac{1}{2}}\bigg\|_{L^2(\mathbb{H}^1)},
	\end{split}
\end{align}
and
\begin{align}\label{4.5}
\begin{split}
	\bigg\|\bigg(\sum_{k_1,k_2}|\zeta_{K}\ast_J \mathcal{L}_{k_2-\ell_2}^{2}\mathcal{L}_{k_1-\ell_1}^1&f
|^2\bigg)^{\frac{1}{2}}\bigg\|_{L^p(\mathbb{H}^1)}\lesssim\bigg\|\bigg(\sum_{k_1,k_2}|\mathcal{L}_{k_2-\ell_2}^{2}\mathcal{L}_{k_1-\ell_1}^1f
|^2\bigg)^{\frac{1}{2}}\bigg\|_{L^p(\mathbb{H}^1)}.
	\end{split}
\end{align}
To prove \eqref{4.4}, we use the group Fourier transform and reduce our problem to estimate the $L^2$ norm of oscillatory integral operator. In that context, the operators in \eqref{fdec} play a crucial role in adjusting each Heisenberg group frequency variables, contributing to the decay estimates of the operator $\zeta_{K}\ast_J \mathcal{L}_{k_2-\ell_2}^2\mathcal{L}_{k_1-\ell_1}^{1}$. To obtain \eqref{4.5}, we apply the bootstrap argument for the vector valued estimate. The bootstrap argument will be explained in Subsection \ref{bstsection}.

\subsection*{Proof of the Proposition \ref{loc} }
We shall prove
\begin{align}\label{p4}
		\|\sup_{K}|\zeta_{K}\ast_J\mathcal{L}_{k_1}^{1,\rm{loc}}f|\|_{L^p(\mathbb{H}^1)}+\|\sup_{K}|\zeta_{K}\ast_J\mathcal{L}_{k_2}^{2,\rm{loc}}\mathcal{L}_{k_1}^{1,\rm{glo}}f| \|_{L^p(\mathbb{H}^1)}\lesssim \|f\|_{L^p(\mathbb{H}^1)}.
\end{align}

\begin{proof}[Proof of (\ref{p4})]
For $\nu=1,2$, recall that  $\mathcal{L}_{j}^\nu f=L_{j}^{\nu}\ast_J f$ and $\mathcal{L}_{k}^{\nu,\rm{loc}}f=\sum_{k}^{\infty}L_{j}^{\nu}\ast_{J} f$ where
\begin{align*}
	L_j^1(y)=\mathcal{F}^{-1}[\varphi(2^j\cdot)](y_1)\delta(y_2)\delta(y_3),\\
	L_j^2(y)=\mathcal{F}^{-1}[\varphi(2^j\cdot)](y_2)\delta(y_1)\delta(y_3).
\end{align*}
For each $a\in\mathbb{R}$, set the diffeomorphism $\mathcal{D}_af(x_1,x_2,x_3)=f(x_1,x_2,x_3+ax_1x_2)$.
With $$\theta_K=(2^{k_1}\cos \theta,2^{k_2}\sin \theta),$$ we use a change of variables for the first variable of $L^1_{k_1}(y)$, where $L^{1,\rm{loc}}_{k_1}(y)=\mathcal{F}^{-1}[\psi(2^{k_1}\cdot)](y_1)\delta(y_2)\delta(y_3)$. Then, the term $ \zeta_{K}\ast_J\mathcal{L}_{k_1}^{1,\rm{loc}}f(x_1,x_2,x_3)$ can be expressed as
\begin{align*}
&\int_{[{\frac{\pi}{4}},{\frac{3\pi}{4}}]}\mathcal{D}_{-2}f(x-\theta_K-\left(\substack{\displaystyle{2^{k_1}z}\\ \\\displaystyle{0}}\right), x_3+2x_1x_2-2^{k_2+2}x_1\sin\theta+2^{k_1+k_2+1}\cos\theta\sin\theta)\mathcal{F}^{-1}{\psi}(z)dz d\theta.
\end{align*}
Denote by $M^1f$ the Hardy-Littlewood maximal function of $f$  with respect to the first variable. Then
by using the rapid decay of $\mathcal{F}^{-1}{\psi}$, we get
\begin{align}\label{lmdo}
\begin{split}
&|\zeta_{K}\ast_J\mathcal{L}_{k_1}^{1,\rm{loc}}f(x_1,x_2,x_3)|\\
&\lesssim			\int_{[{\frac{\pi}{4}},{\frac{3\pi}{4}}]} M^{1}\mathcal{D}_{-2}f(x_1,x_2-2^{k_2}\sin\theta, x_3+2x_1x_2-2^{k_2+2}x_1\sin\theta+2^{k_1+k_2+1}\cos\theta\sin\theta) d\theta\\
&=		\int_{[{\frac{\pi}{4}},{\frac{3\pi}{4}}]} \mathcal{D}_4M^{1}\mathcal{D}_{-2}f(x_1,x_2-2^{k_2}\sin\theta, x_3-2x_1x_2 +2^{k_1+k_2+1}\cos\theta\sin\theta) d\theta.
\end{split}
\end{align}
Then one can see that 
\begin{align*}
&\sup_{K}|\zeta_{K}\ast_J\mathcal{L}_{k_1}^{1,\rm{loc}}f(x_1,x_2,x_3)|\lesssim \mathcal{M}\mathcal{D}_4M^{1}\mathcal{D}_{-2}f(x_1,x_2,x_3-2x_1x_2),
\end{align*}
where  the maximal operator $\mathcal{M}$ is defined by
\begin{align*}
\mathcal{M}f(x_1,x_2,x_3)=	\sup_{k_2,k_3\in \bZ}\int_{[{\frac{\pi}{4}},{\frac{3\pi}{4}}]} |f(x_1,x_2-2^{k_2}\sin\theta, x_3-2^{k_3}\cos\theta \sin\theta )|d\theta.
\end{align*}
Applying Theorem $3.2$ in \cite{AIF_1992__42_3_637_0}, we deduce that the operator $\mathcal{M}$ is  bounded on $L^p(\mathbb{R}^3)$.
Thus, we get \begin{align*}
		\|\sup_{K}|\zeta_{K}\ast_J\mathcal{L}_{k_1}^{1,\rm{loc}}f|\|_{L^p(\mathbb{H}^1)} \lesssim \|f\|_{L^p(\mathbb{H}^1)}.
\end{align*}
From \(\mathcal{L}_{k_1}^{1,\rm{glo}}f = f - \mathcal{L}_{k_1}^{1,\rm{loc}}f\) and a similar approach as described above, we obtain the \( L^p \) boundedness of the operator $f \mapsto \sup_{K} | \zeta_{K} \ast_J \mathcal{L}_{k_2}^{2,\mathrm{loc}} \mathcal{L}_{k_1}^{1,\mathrm{glo}} f |.$
\end{proof}

\section{Proof of the Proposition \ref{3dcy} for $p=2$}\label{dcysection}
In this section, we aim to prove the following estimate: 
\begin{align*}
	\bigg\|\bigg(\sum_{k_1,k_2}|\zeta_{K}\ast_J \mathcal{L}_{k_2-\ell_2}^{2}\mathcal{L}_{k_1-\ell_1}^1&f
|^2\bigg)^{\frac{1}{2}}\bigg\|_{L^2(\mathbb{H}^1)}\lesssim 2^{-\varepsilon(\ell_1+\ell_2)}\bigg\|\bigg(\sum_{k_1,k_2}|\mathcal{L}_{k_2-\ell_2}^{2}\mathcal{L}_{k_1-\ell_1}^1f
|^2\bigg)^{\frac{1}{2}}\bigg\|_{L^2(\mathbb{H}^1)}.
\end{align*}
To achieve this, it is enough to show that the operator norm $$\|\zeta_{K}\ast_J \mathcal{L}_{k_2-\ell_2}^{2}\mathcal{L}_{k_1-\ell_1}^1\|_{L^2\rightarrow L^2}\lesssim 2^{-\varepsilon(\ell_1+\ell_2)}$$ which directly implies the above estimate. By  Proposition \ref{H2}, it suffices to show the existence of $C, \varepsilon>0$ such that  $$ \|\widehat{\zeta_K}(\lambda) \cdot \widehat{\mathcal{L}^2_{k_2-\ell_2}}(\lambda)\cdot \widehat{\mathcal{L}_{k_1-\ell_1}^{1}}(\lambda)\|_{L^2(\mathbb{R}^1)\rightarrow L^2(\mathbb{R}^1)}\le C2^{-\varepsilon (\ell_1+\ell_2)}\ \text{uniformly in $\lambda\ne 0$}$$
For this purpose, we use the change of variables and rewrite the measure $\zeta_K$ in (\ref{022m}) as
\begin{align*}
	\langle \zeta_K,f \rangle=\frac{1}{2^{k_1}}\int_{|y_1|<2^{k_1-1/2}} f\left(y_1, 2^{k_2}\sqrt{1-|2^{-k_1}y_1|^2},0\right) \frac{dy_1}{\sqrt{1-|2^{-k_1}y_1|^2}} \ \text{for $f\in \mathcal{S}(\mathbb{R}^3)$}.
\end{align*}
By taking the nonnegative smooth cutoff function $\eta(y_1)$ supported $|y_1|\le 3/4$ and the Dirac mass $\delta(y_2),\delta(y_3)$ on the real lines $\mathbb{R}$, we  replace the measure $\zeta_K$ with the  following explicit  form:
$$ \zeta_K(y_1,y_2,y_3)= 2^{-k_1}\eta(2^{-k_1}y_1) \delta\left(y_2- 2^{k_2}\sqrt{1-|2^{-k_1}y_1|^2}\right)\delta(y_3).$$
  By computing the kernel for the group Fourier transform of $\zeta_K$ in \eqref{gf}, we write  
\begin{align*}
\widehat{\zeta_k}(\lambda)\phi(x)=		 \int e^{-2\pi i {\lambda} 2^{k_2-1}(x+y)\sqrt{1-\{2^{-k_1}(x-y)\}^2} }2^{-k_1}{\eta}(2^{-k_1}(x-y)){\phi(y)}dy,
\end{align*}
 Define $\beta\in C^{\infty}$ supported on $|x|\approx 1$ satisfying $\eta(x)=\sum_{s=0}^{\infty}\beta(2^sx)$.  Then we can decompose $\widehat{\zeta_k}(\lambda)=\sum_{s=0}^{\infty}\widehat{\zeta_K^s}(\lambda)$ where
\begin{align}\label{sdec} 
[\widehat{\zeta_K^s}(\lambda)]\phi(x)=	\int e^{-2\pi i {\lambda} 2^{k_2-1}(x+y)\sqrt{1-\{2^{-k_1}(x-y)\}^2} }2^{-k_1}\beta(2^{s-k_1}(x-y))\phi(y)dy.
\end{align}  
Hence, it suffices to show the existence of $C, \varepsilon>0$ independent of $\lambda\ne 0$ such that  
\begin{align}\label{dcy1} \|\widehat{\zeta_K^s}(\lambda) \cdot \widehat{\mathcal{L}^2_{k_2-\ell_2}}(\lambda)\cdot \widehat{\mathcal{L}_{k_1-\ell_1}^{1}}(\lambda)\|_{L^2(\mathbb{R}^1)\rightarrow L^2(\mathbb{R}^1)}\le C2^{-\varepsilon (s+\ell_1+\ell_2)}\ \text{}.
\end{align}
To simplify the calculation of   the above operator norm  on $L^1(\mathbb{R})$, we apply the dilation $x \mapsto 2^{k_1}x$ and $y \mapsto 2^{k_1}y$ on the kernel of $ \widehat{\zeta_K^s}(\lambda) \cdot \widehat{\mathcal{L}^2_{k_2-\ell_2}}(\lambda) \cdot \widehat{\mathcal{L}_{k_1-\ell_1}^{1}}(\lambda) $ to see that
$$\| \widehat{\zeta_K^s}(\lambda) \cdot \widehat{\mathcal{L}^2_{k_2-\ell_2}}(\lambda) \cdot \widehat{\mathcal{L}_{k_1-\ell_1}^{1}}(\lambda) \|_{L^2(\mathbb{R}^1)\rightarrow L^2(\mathbb{R}^1)}=\|\mathcal{C}^s_K \mathcal{P}_{k_1+k_2-\ell_2}^{2} \mathcal{P}_{\ell_1}^1\|_{L^2(\mathbb{R}^1)\rightarrow L^2(\mathbb{R}^1)}$$ 
where the three consecutive operators are defined as
\begin{align*}
&\mathcal{C}^s_K \phi(x)  := \int e^{-2\pi i \lambda 2^{k_1+k_2} \Phi(x,y)} \beta(2^s(x-y))  \phi(y)\, dy
\end{align*}
for $\Phi(x,y) =  (x+y)\sqrt{1-(x-y)^2}$ and 
\begin{align*}
\mathcal{P}_{k_1+k_2-\ell_2}^{2} \phi(y) &: = \varphi(2^{k_1+k_2-\ell_2}\lambda y) \phi(y),\\
\mathcal{P}_{\ell_1}^1\phi(y) &: = \int e^{2\pi i \xi y} \varphi(\xi 2^{-\ell_1}) \mathcal{F}\phi(\xi)d\xi
\end{align*}
where     $\mathcal{F}$ stands for the Fourier transform in the Euclidean space $\mathbb{R}^1$, and   $\lambda$ is omitted for simplicity.  Hence
the decay estimate \eqref{dcy1} reduces to proving the following estimate:
\begin{align}\label{dcy2}
    \|\mathcal{C}^s_K \mathcal{P}_{k_1+k_2-\ell_2}^{2} \mathcal{P}_{\ell_1}^1\|_{L^2(\mathbb{R})\mapsto L^2(\mathbb{R})} \lesssim 2^{- (s+\ell_1+\ell_2)/4}.
\end{align}
\subsection{Proof of (\ref{dcy2}) for  $2^{-\ell_2/4}$ decay}
Note the following van der Corput type lemma first.
\begin{lemma}\label{lpt}
Let $k\ge 1$. Suppose that $\phi$ is a real smooth function in $(a,b)$ such that
$$|\phi^{(k)}(x)|\ge 1\ \text{ for all $x\in (a,b)$}$$
Moreover assume that $\phi''$ changes its sign at most $B$ times if $k=1$ above. Then
$$\left|\int_a^b e^{i\lambda \phi(x)}\psi(x)dx\right|\le c_k \lambda^{-1/k}\left[ |\psi(b)|+\int_a^b|\psi'(x)|dx \right] $$
where the bound $c_k$   is independent of $\lambda$ and $k$ if $k\ge 2$, but depend on $B$ when $k=1$.
\end{lemma}
See the page 334 in \cite{MR1232192}. In order to prove (\ref{dcy2}), we first claim that
\begin{proposition}\label{prop45}
Given $s>0$ and $m\in\mathbb{Z}$, consider the operator $\mathcal{C}_{K}^{s,m}$ defined by $$	\mathcal{C}_{K}^{s,m}\phi(x):= \int e^{-2\pi i{\lambda}2^{k_1+k_2}{\Phi}(x,y)} \varphi(2^{-m} y)\beta(2^{s}(x-y)){\phi}(y)dy.      $$
Then it holds that
\begin{align}\label{044}
\| \mathcal{C}^{s,m}_K \|_{L^2(\mathbb{R})\mapsto L^2(\mathbb{R})} \le C2^{-s/4} \min\{ |\lambda 2^{k_1+k_2} 2^m|^{-1/4}, |\lambda 2^{k_1+k_2} |^{-1/4}\}.
\end{align}
\end{proposition}

\begin{proof}[Proof of Proposition \ref{prop45}]
We consider the two cases $2^m\le  2^{-s+10}$ and $2^m>  2^{-s+10}$.\\
{\bf Case 1}.
Let $2^m\le  2^{-s+10}$. For this case, one can see from $|x-y|\sim 2^{-s}$ and $  |y|\le  2^{-s+10}$ that the support of our integral kernel of  $\mathcal{C}_{K}^{s,m}$ is contained in $|x+y|\lesssim 2^{-s+10}$. This region  incldues the singular points $x+y=0$ where   the hessian $\Phi_{xy}(x,y)= \frac{x+y}{(1-|x-y|^2)^{3/2}} $  vanishes. By splitting $|x+y|\lesssim 2^{-s}$ into the small pieces $|x+y|\approx 2^{-s}2^{-j}$ over $j\ge -10$,   decompose  
$\mathcal{C}^{s,m}_K \phi(x)=\sum_{j=-10}^\infty\mathcal{C}^{s}_{j}\phi(x)$ where
\begin{align}\label{kernel}
		\mathcal{C}^{s}_{j}\phi(x)&:=\int e^{-2\pi i{\lambda}2^{k_1+k_2}{\Phi}(x,y)}\varphi(2^{-m}y)\beta(2^{s}(x-y)) \varphi(2^{j+s}(x+y)) {\phi}(y)dy.
	\end{align}
To show (\ref{044}) for the case $2^m\le 2^{-s+10}$, it suffices to prove that
\begin{align}\label{042}
\| \mathcal{C}^{s}_{j} \|_{L^2(\mathbb{R})\mapsto L^2(\mathbb{R})} \le 2^{10} 2^{-(j+s)/4}   |\lambda 2^{k_1+k_2} |^{-1/4}.
\end{align}
To show (\ref{042}),  we split the support of integral kernel in \eqref{kernel} into the two parts $$\text{$x+y>0\ \text{and}\ x+y<0$.}$$ It suffices to treat the one region $x+y>0$.  Then by absorbing $\varphi(2^{-m}y)\phi(y) $ into $\phi(y)$ and   compute
  the integral kernel $K(x,z)$ of $ \mathcal{C}^{s}_{j} [  \mathcal{C}^{s}_{j} ]^*$ as
\begin{align*}
K(x,z)&= \int e^{i\lambda 2^{k_1+k_2}  [\Phi(x,y)-\Phi(z,y)]} h(x,y,z)dy \times \chi_{[-2^{-(j+s)}, 2^{-(j+s)}]}((x-z)),
\end{align*}
where   the amplitude $h(x,y,z)$ is given by
\begin{align*}
&h(x,y,z)=\varphi(2^{j+s}(x+y))\varphi(2^{j+s}(z+y)\beta(2^s(x-y))\beta(2^s(z-y)).
\end{align*}
Here $x+y,z+y$ are restricted to positive numbers. Note that for a fixed $x,z$, 
\begin{itemize}
\item[(1)] the cutoff function $h(x,y,z)$ is supported on an interval $I\subset [-2^{-j-s},2^{-j-s}]-x$.
\item[(2)]  $\int_I |\partial_y h(x,y,z)|dy\le 2^{3}$  because $|\partial_y (\varphi(2^{j+s}(x+y))\varphi(2^{j+s}(z+y))|\le 2^{j+s+2}.$
\item[(3)] $\partial_{y}^2 [\Phi(x,y)-\Phi(z,y)]$ changes its sign at most twice.
\end{itemize}
By the support condition  
\begin{align*}
[\Phi]_{xy}(x,y)=( x+y)\left( 1- (x-y)^2 \right)^{-3/2}\approx 2^{-j-s}.
\end{align*}
By this with the  support condition $x+y\approx 2^{-j-s},z+y\approx 2^{-j-s}$, where both $x+y,z+y$ are positive, one can apply   the mean value theorem to find $c\in (0,1)$ depending  on $x,y,z$ such that
\begin{align*}
|\partial_y [\Phi(x,y)-\Phi(z,y)]|=|[\Phi]_{xy}(z+c(x-z),y)(x-z)|\ge 2^{-j-s-10}|x-z|.
\end{align*}
By this lower bound of the $y$-derivative combined with the properties (1),(2) and (3)  above, one can utilize Lemma \ref{lpt}  for $k=1$ to obtain that
\begin{align}\label{nee}
\left| \int e^{i\lambda 2^{k_1+k_2}  [\Phi(x,y)-\Phi(z,y)]} h(x,y,z)dy \right| &\le 2^{10} \min\{(\lambda 2^{k_1+k_2}2^{-(j+s)}|x-z|)^{-1}, 2^{-j-s}\}\nonumber\\
&\le 2^{10}(\lambda 2^{k_1+k_2} |x-z|)^{-1/2}. 
\end{align}
Therefore 
\begin{align*}
\int_{|x-z|\le 2^{-(j+s)}} |K(x,z)| dx\ \text{or}\ dz \le  2^{10}(\lambda 2^{k_1+k_2})^{-1/2} 2^{-(j+s)/2}.
\end{align*}
This leads a desired bound of  $ \|\mathcal{C}^{s}_{j} [  \mathcal{C}^{s}_{j} ]^*\|_{op}$ showing (\ref{042}).  We are done with the Case 1.\\
{\bf Case 2}.  Let $2^{m}\ge 2^{-s+10}$. Then
the integral kernel $K(x,z)$ of $ \mathcal{C}^{s,m}_{K} [  \mathcal{C}^{s,m}_{K} ]^*$ as
\begin{align*}
K(x,z)&= \int e^{i\lambda 2^{k_1+k_2}  [\Phi(x,y)-\Phi(z,y)]} h(x,y,z)dy \times \chi_{[-2^{-s}, 2^{-s}]}(x-z)
\end{align*}
where
\begin{align*}
&h(x,y,z)=\varphi(2^{-m}y)\beta(2^s(x-y))\beta(2^s(z-y)).
\end{align*}
Note that for a fixed $x,z$,
\begin{itemize}
\item[(1)] the cutoff function $h(x,y,z)$ is supported on an interval $I\subset [-2^{-s},2^{-s}]+x$.
\item[(2)]  $\int_I |\partial_y h(x,y,z)|dy\le 2^{3}$ follows from the estimate $|\partial_y (\varphi(2^{s}(x+y))\varphi(2^{s}(z+y))|\le 2^{s+2}$.
\item[(3)] $\partial_{y}^2 [\Phi(x,y)-\Phi(z,y)]$ does not change its sign.
\end{itemize}
By the support condition  
\begin{align*}
[\Phi]_{xy}(x,y)=( x+y)\left( 1- (x-y)^2 \right)^{-3/2}\approx y \approx 2^m
\end{align*}
  with   $x\approx z\approx y\approx 2^m\gg 2^s\approx |x-y|,|z-y|$, one can apply   the mean value theorem to find $c\in (0,1)$ depending  on $x,y,z$ such that
\begin{align*}
|\partial_y [\Phi(x,y)-\Phi(z,y)]|=|[\Phi]_{xy}(z+c(x-z),y)(x-z)|\ge 2^{m-2}|x-z|.
\end{align*}
By this lower bound of the $y$-derivative combined with the properties (1),(2) and (3)  above, one can apply Lemma \ref{lpt} for $k=1$ to obtain that
\begin{align*}
\left| \int e^{i\lambda 2^{k_1+k_2}  [\Phi(x,y)-\Phi(z,y)]} h(x,y,z)dy \right| &\le 2^{10} \min\{(\lambda 2^{k_1+k_2}2^{m}|x-z|)^{-1}, 2^{-s}\}\\
&\le 2^{10}(\lambda 2^{k_1+k_2}2^{m} |x-z|)^{-1/2} 2^{-s/2}
\end{align*}
Therefore 
\begin{align*}
\int_{|x-z|\le 2^{-s+1}} |K(x,z)| dx\ (\text{or}\ dz) \le  2^{10}(\lambda 2^{k_1+k_2}2^m)^{-1/2} 2^{-s},
\end{align*}
which implies the first part of (\ref{044}) and the second part follows from $2^m\le 2^{-s+10}$.
\end{proof}
Recall $\mathcal{P}_{k_1+k_2-\ell_2}^{2}\phi(y)=\varphi\left(\lambda 2^{k_1+k_2-\ell_2}y \right)\phi(y)$, which is to be written as $\varphi(2^{-m}y)\phi(y)$, and apply Proposition \ref{prop45}. Then we have
\begin{align}\label{ssp}
\|\mathcal{C}^s_K \mathcal{P}_{k_1+k_2-\ell_2}^{2} \|_{L^2(\mathbb{R})\mapsto L^2(\mathbb{R})}  \le C 2^{-s/4}\min\{2^{-  \ell_2/4}, |\lambda 2^{k_1+k_2}|^{-1/4}\}.
\end{align}
\subsection{Proof of (\ref{dcy2}) for  $2^{-\ell_1/2}$ decay}
In view of (\ref{ssp}),   to prove (\ref{044}), it suffices to prove that
\begin{align}\label{dcc}
\|\mathcal{C}^s_K \mathcal{P}_{k_1+k_2-\ell_2}^{2} \mathcal{P}^1_{\ell_1}\|_{L^2(\mathbb{R})\mapsto L^2(\mathbb{R})}  \lesssim 2^{-\ell_1/2}\ \text{if}\ 2^{\ell_1}\ge 2^{10}\left( 2^{\ell_2}+|\lambda 2^{k_1+k_2}|\right).
\end{align}
To prove (\ref{dcc}), recall
$\Phi(x,y)=(x+y)\sqrt{1-(x-y)^2}$
and denote the kernel:
\begin{align*}
L(x,z)=\left[\int e^{2\pi i [\lambda 2^{k_1+k_2} \Phi(x,y)-\xi( y-z)]}  \varphi(2^{k_1+k_2-\ell_2} \lambda y) \beta(2^s (x-y))   \varphi(2^{-\ell_1}\xi) dyd\xi\right] 
\end{align*}
to express
\begin{align*}
&\mathcal{C}^s_K \mathcal{P}_{k_1+k_2-\ell_2}^{2}\mathcal{P}_{\ell_1}^1\phi(x)=\int L(x,z) \phi(z)dz.
\end{align*}
For $m\in\mathbb{Z}$, we denote 
\begin{align*}
	C_m&:=\{x\in \bR: m\cdot2^{-s}\leq |x|<(m+1)\cdot2^{-s}\},\\
	C'_m&:=\{x\in \bR:(m-5)\cdot2^{-s}\leq |x|<(m+5)\cdot2^{-s}\},
\end{align*}
and define the amplitudes as
\begin{align*}
	G_m(x,z)&:=L(x,z)\chi_{C_m}(x)\chi^c_{{C}'_m}(z),\\
	B_m(x,z)&:=L(x,z)\chi_{C_m}(x)\chi_{{C}'_m}(z).
\end{align*}
We define the good and bad operator having the above kernels respectively as
\begin{align*}
 \mathcal{G}_m\phi(x)&:=\int G_m(x,z)\phi(z)dz,\\
\mathcal{B}_m\phi(x)&:=\int B_m(x,z)\phi(z)dz.
\end{align*}
Then one can decompose
\begin{align}\label{ee1}
&\mathcal{C}^s_K \mathcal{P}_{k_1+k_2-\ell_2}^{2}\mathcal{P}_{\ell_1}^1\phi(x)= \sum_{m\in \mathbb{Z}} \mathcal{G}_m\phi(x) + \sum_{m\in\mathbb{Z}}  \mathcal{B}_m\phi(x).
\end{align}
\begin{proposition} Suppose that the two operators $\mathcal{G}_m$ and $\mathcal{B}_m$ are defined above. Then
\begin{align}
\label{ogp}
		&\bigg\|\sum_{m\in \bZ}\mathcal{G}_{m}\bigg\|_{L^2(\mathbb{R}^1)\mapsto L^2(\mathbb{R}^1)}\leq C2^{-\ell_1},\\
		\label{obp}
		&\bigg\|\sum_{m\in \bZ}\mathcal{B}_{m}\bigg\|_{L^2(\mathbb{R}^1)\mapsto L^2(\mathbb{R}^1)}\leq C2^{-\ell_1/2},
	\end{align}
	for some $C>0$ independent of $\ell_1$.
	\end{proposition}
This yields the main estimate \eqref{dcc} of this subsection.
\begin{proof}[Proof of \eqref{ogp}]
In the kernel $L(x,z)$, observe $$\int e^{-2\pi i \xi(y-z)} \varphi(2^{-\ell_1}\xi) d\xi = 2^{\ell_1}\widehat{\varphi}(2^{\ell_1}(y-z)) =O(2^{-\ell_1}|y-z|^{-2})$$ By this with the support condition
 $|y-z|\approx |x-z|$ on the support   of $G_m(x,z)$, one can obtain that 
\begin{align*}
|G_m(x,z)|&\le \chi_{C_m}(x) \chi^c_{{C}'_m}(z)\int  \left|  \beta(2^s (x-y))  2^{\ell_1} \widehat{\varphi}(2^{\ell_1}(y-z))\right|  dy\\
&\le    \frac{\chi_{C_{m}}(x)2^{-s}}{2^{\ell_1}|x-z|^2} \psi^c\bigg(\frac{|x-z|}{2^{-s}}\bigg).
\end{align*}
Thus
$$\sum_{m\in\mathbb{Z}}   |G_m(x,z)|\le     \frac{2^{-s+5}}{2^{\ell_1}|x-z|^2} \psi^c\bigg(\frac{|x-z|}{2^{-s}}\bigg)$$
So,
\begin{align}\label{514}
\begin{split}
 \int \sum_{m\in\bZ} |G_m(x,z)|dx\ (\text{or}\ dz)\le 2^{10} 2^{-{\ell_1}}.
\end{split}
\end{align}
By Schur's lemma, this yields \eqref{ogp}. 
\end{proof}

\begin{proof}[Proof of \eqref{obp}]
By the localization principle, we have
\begin{align*} 
	\bigg\|\sum_{m\in\bZ}\mathcal{B}_m\bigg\|_{L^2\rightarrow L^2}\le  2^5 \sup_{m\in \bZ}\|\mathcal{B}_m\|_{L^2\rightarrow L^2}.
\end{align*}
To estimate $\|\mathcal{B}_m\|_{L^2\rightarrow L^2}$, denote $\phi_m(z)=\chi_{C_m'}(z)\phi(z)$ and write
\begin{align*}
	\mathcal{B}_m\phi(x)&=\int B_m(x,z)\phi_m(z)dz=\int P_m(x,\xi)\mathcal{F}{\phi_m}(\xi)d\xi\ \text{where}\\
	P_m(x,\xi)&=\chi_{C_m}(x)	\varphi({2^{-\ell_1}\xi})\int e^{-2\pi i (\lambda 2^{k_1+k_2}\Phi(x,y)-\xi y)} \varphi(2^{k_1+k_2-\ell_2} \lambda y) \beta(2^s (x-y))  dy 
	\end{align*}
By using the condition (\ref{dcc}) with $\lambda 2^{k_1+k_2} y\approx 2^{\ell_2}$ and $|x-y|\approx 2^{-s}\le 1/2$,
\begin{align} \label{k42}
	|\partial_y(\lambda 2^{k_1+k_2}\Phi(x,y)-\xi y)|=\bigg|\lambda 2^{k_1+k_2}\frac{(1-2y(y-x))}{\sqrt{1-(x-y)^2}}-\xi\bigg|\ge |\xi|/2\approx 2^{\ell_1}.
\end{align}
By this combined with the property
  for a fixed $x$,
\begin{itemize}
\item[(1)] the cutoff function $\varphi(2^{k_1+k_2-\ell_2} \lambda y) \beta(2^s (x-y))  $  supported on $I= [-2^{-s},2^{-s}]+x$.
\item[(2)]  $\int_I |\partial_y \varphi(2^{k_1+k_2-\ell_2} \lambda y) \beta(2^s (x-y))|dy\le 2^{3}.$  
\item[(3)]  $y\rightarrow \partial_y^2\Psi_{K}(x,y)$ changes its sign at most twice. 
\end{itemize}
  We are able to use the corollary on the page 334 of \cite{MR1232192} to obtain that
$$\left|\int e^{-2\pi i (\Psi_{K}(x,y)-\xi y)} \varphi(2^{k_1+k_2-\ell_2} \lambda y) \beta(2^s (x-y))  dy \right|\le \frac{2^{10}}{2^{\ell_1}}.$$
This with the support condition $\chi_{C_m}(x)\varphi(2^{-\ell_1}\xi)$, 
\begin{align*}
	&\int|P_m(x,\xi)|dx\lesssim {2^{-\ell_1}},\\
	&\int|P_m(x,\xi)|d\xi\lesssim {1}.
\end{align*}
Therefore we prove (\ref{obp}).
\end{proof}

\section{Proof of Proposition \ref{3dcy} and Bootstrap argument}\label{bstsection}
Recall that $\zeta_{K}\ast_J f=\zeta_{K}\ast_J\mathcal{L}_{k_1}^{1,\rm{loc}}f+\zeta_{K}\ast_J\mathcal{L}_{k_2}^{2,\rm{loc}}\mathcal{L}_{k_1}^{1,\rm{glo}}f+\sum_{\ell_1,\ell_2=0}^{\infty}\zeta_{K}\ast_J \mathcal{L}_{k_2-\ell_2}^2\mathcal{L}_{k_1-\ell_1}^{1}f.$
By Proposition \ref{loc}, we have the $L^p$ estimates for $\zeta_{K}\ast_J\mathcal{L}_{k_1}^{1,\rm{loc}}f+\zeta_{K}\ast_J\mathcal{L}_{k_2}^{2,\rm{loc}}\mathcal{L}_{k_1}^{1,\rm{glo}}f$. It suffices to show
\begin{align}\label{elldcy}
	\|\sup_{K\in \bZ^2}|\zeta_{K}\ast_J \mathcal{L}_{k_2-\ell_2}^2\mathcal{L}_{k_1-\ell_1}^{1}f|\|_{L^p(\mathbb{H}^1)}\lesssim 2^{-\varepsilon_p(\ell_1+\ell_2)}\|f\|_{L^p(\mathbb{H}^1)}.
\end{align}
Using interpolation argument and Lemma \ref{lpt}, the estimate \eqref{elldcy} follows from 
\begin{align}\label{417}
	\begin{split}
	\bigg\|\bigg(\sum_{k_1,k_2}|\zeta_{K}\ast_J \mathcal{L}_{k_2-\ell_2}^{2}\mathcal{L}_{k_1-\ell_1}^1&f
|^2\bigg)^{\frac{1}{2}}\bigg\|_{L^2(\mathbb{H}^1)}\lesssim 2^{-\varepsilon(\ell_1+\ell_2)}\bigg\|\bigg(\sum_{k_1,k_2}|\mathcal{L}_{k_2-\ell_2}^{2}\mathcal{L}_{k_1-\ell_1}^1f
|^2\bigg)^{\frac{1}{2}}\bigg\|_{L^2(\mathbb{H}^1)}
	\end{split}
\end{align}
and
\begin{align}\label{415}
\begin{split}
	\bigg\|\bigg(\sum_{k_1,k_2}|\zeta_{K}\ast_J \mathcal{L}_{k_2-\ell_2}^{2}\mathcal{L}_{k_1-\ell_1}^1&f
|^2\bigg)^{\frac{1}{2}}\bigg\|_{L^p(\mathbb{H}^1)}\lesssim\bigg\|\bigg(\sum_{k_1,k_2}|\mathcal{L}_{k_2-\ell_2}^{2}\mathcal{L}_{k_1-\ell_1}^1f
|^2\bigg)^{\frac{1}{2}}\bigg\|_{L^p(\mathbb{H}^1)}.
	\end{split}\end{align}
	In the previous section, we proved \eqref{417}.
Combining Proposition \ref{loc} and decay estimate \eqref{elldcy}, we arrive at 
\begin{align}\label{cus}
	\|\sup_{K\in \bZ^2}|\zeta_{K}\ast_J f|\|_{L^p(\mathbb{H}^1)}\lesssim 
	\|f\|_{L^p(\mathbb{H}^1)}.
\end{align}
In the spirit of Nagel, Stein and Wainger \cite{MR466470}, we utilize the following lemma.
\begin{lemma}\label{3BST}
	If $\|\sup_{K\in \bZ^2}|\zeta_{K}\ast_J f|\|_{L^p(\mathbb{H}^1)}\le C_1 
	\|f\|_{L^p(\mathbb{H}^1)}$ , $\|\zeta_K\ast_Jf\|_{L^r(\mathbb{H}^1)}\leq C_2\|f\|_{L^r(\mathbb{H}^1)}$ for $1<r<\infty$, 
\begin{align}\label{bst3}
			\bigg\|\bigg(\sum_{K\in \bZ^2}|\zeta_{K}\ast_J f_{K}|^2\bigg)^{\frac{1}{2}}\bigg\|_{L^q(\mathbb{H}^1)}\le C \bigg\|\bigg(\sum_{K\in \bZ^2}| f_{K}|^2\bigg)^{\frac{1}{2}}\bigg\|_{L^q(\mathbb{H}^1)}.
\end{align}
hold for all $q$ with $\frac{1}{q}<\frac{1}{2}(1+\frac{1}{p})$, where \(C > 0\) is a constant that depends only on \(p\) and $C_1$, $C_2$.
\end{lemma}

After confirming $\|\sup_{K\in \mathbb{Z}^2}|\zeta_{K}\ast_J f|\|_{L^2(\mathbb{H}^1)}\lesssim 
	\|f\|_{L^2(\mathbb{H}^1)}$, we apply Lemma \ref{3BST} to deduce \eqref{bst3} for $q>\frac{4}{3}$. Letting $\{f_{K}\}=\{\mathcal{L}_{k_2-\ell_2}^{2}\mathcal{L}_{k_1-\ell_1}^1f\}$, we obtain \eqref{415} for $p>\frac{4}{3}$. Combining this with \eqref{417} gives the decay estimate \eqref{elldcy} for $p>\frac{4}{3}$, hence yielding \eqref{cus} for a wider range of $p$. By repeating this process using Lemma \ref{3BST}, and iterating sufficiently, we prove \eqref{cus} for all $p>1$, thus completing the proof of Theorem \ref{3mt1}.

\section{Reduction and the Proof of Theorem \ref{3mt0}; Unboundedness}\label{Asect}
Recall that the convolution of $f$ and $g$ by $f\ast_J g$ on the Heisenberg group means $		f\ast_J g(x):=\int_{\bR^d}f(x\cdot_Jy^{-1})g(y)dy.
$ Given matrix $A\in \mathbb{M}_{d}(\bR)$, we denote $
	A_s=\frac{A+A^t}{2},  A_w=\frac{A-A^t}{2}.
$ Then we have
\begin{align}\label{mdec}
	x^tAy=\frac{1}{2}{y}^tA_{s} {y}+ {x}^tA_{w} {y}+\frac{1}{2}  {x}^tA_{s} {x}
	-\frac{1}{2} ({x}- {y})^tA_{s}({x}- {y}) \text{ for $x,y\in \mathbb{R}^d$.}
\end{align} 

\begin{lemma}\label{l1}
Let $\nu$ be a finite, compactly supported Borel measure on $\mathbb{R}^d$ and let $Tf(x,x_{d+1})=\int_{\mathbb{R}^{d}}f(x-y,x_{d+1}-x^tAy) d\nu(y) $. Then
for $\mathcal{D}_{A_s}f(x,x_{d+1})=f( {x},x_{d+1}-\frac{1}{2} {x}^tA_s  {x})$, it holds that if $A_w=0$,
\begin{align}\label{pp2}
Tf(x,x_{d+1})=  \int_{\mathbb{R}^d} \mathcal{D}_{A_s}f\left((x,x_{d+1}+ \frac{1}{2} {x}^tA  {x})-(y,\frac{1}{2}y^tA_sy)\right)d \nu (y),
\end{align}
and if $A_w\ne 0$, 
\begin{align}\label{pp1} 
Tf(x,x_{d+1}) &= 
 \int_{\mathbb{R}^d} \mathcal{D}_{A_s}f\left((x,x_{d+1}+\frac{1}{2} {x}^tA  {x})-(y,\frac{1}{2}y^tA_sy+x^tA_wy)\right)d \nu (y) 
\end{align}
for a measurable function $f$ on $\bR^{d+1}$.  
\end{lemma}
\begin{proof}
This follows from the identity in \eqref{mdec}.
\end{proof}
Consider the case where $A_w\ne 0$ above; that is, when $a_{12}-a_{21}\ne 0$. Since $A_w$ is a scalar multiple of a skew-symmetric matrix $J$, we can simplify our analysis by considering the special case where $A_w=J$, which is the Heisenberg group $\mathbb{H}^1$.

\subsection*{Proof of Theorem \ref{3mt0}}
For the one parameter necessity proof, we shall prove the unboundedness of the $L^p$ norm for the operator $\mathcal{E}^1_{A}$ associated with the identity matrix $A=I$.  Write $	E^I_{2^k,2^k}f(x,x_3)=\int_{[0,2\pi]}f(x_1-2^{k}\cos\theta,x_2-2^{k}\sin\theta,x_3-2^{k}x_1\cos\theta-2^{k}x_2\sin\theta)d\theta.$ 
By taking $f(x,x_3)=g(x,x_3-\frac{x_1^2+x_2^2}{2})$, one can express
\begin{align*}
	E^I_{2^k,2^k}f(x,x_3)&=\int_{[0,2\pi]}g(x_1-2^{k}\cos\theta,x_2-2^{k}\sin\theta,x_3-\frac{x_1^2+x_2^2}{2}-2^{2k-1})d\theta\\
	&=:\tilde{E}_{2^k}g(x,x_3-\frac{x_1^2+x_2^2}{2})
\end{align*}
Then $\|E^I_{2^k,2^k}\|_{L^2\mapsto L^2}=\|\tilde{E}_{2^k}\|_{L^2\mapsto L^2}$.
Let $h_{\delta}(x,x_3)=\chi_{B_{10}}(x)\chi_{[0,\delta]}(x_3)$ where $B_r=\{x\in R^2: |x|<r\}$. Then $\tilde{E}_{2^k}h_{\delta}(x,x_3)=1$ on $B_{1}\times{A_k^{\delta}} $ where $A_s^{\delta}=\{x_3\in R: 2^{2s}\leq x_3\leq 2^{2s}+\delta\}$. We can determine an integer $m\in \mathbb{N}$ such that $2^{-2m}<|\delta|<2^{-2m+2}$. Consequently, the sets $A_s^{\delta}$ become disjoint for $s>-m+1$. Then we have $\|\sup_{-m<k<1}|\tilde{E}_{2^k}h|\|_{p}\approx (m\delta)^{1/p}\approx (\delta|\log\delta|)^{1/p}$ and $\|h\|_{p}\approx \delta^{1/p}$. Choosing $\delta\rightarrow0$, we can check that the operator norm of $\mathcal{E}_{I}^1$ is unbounded for all $0<p<\infty$. To treat $\mathcal{E}_{A}^2$, it suffices to consider $E_{2^{k+2a},2^{k}}^{A}$ for $A = \begin{pmatrix} c & 0 \\ 0 & c \cdot 2^{2a} \end{pmatrix}$. Following a similar process as above, we omit the remaining proof for $\mathcal{E}_{A}^2$.

\section{Proof of Theorem \ref{3mt2}; General matrix $A$}  \label{GAsec}
\subsection{The non-symmetric cases $A_w\ne 0$}
As we mentioned in Section \ref{Asect}, let $A_w=J$. Suppose $A_s=\begin{pmatrix}
b&e\\e&d	
\end{pmatrix}.$ In view of \eqref{pp1}, we first observe that
\begin{align*}
	 E^A_{2^{k_1}, 2^{k_2}} f(x, x_3)= \mu_{k_1, k_2}^{b, d, e} \ast_J \mathcal{D}_{A_s}f(x, x_3+ \frac{1}{2} {x}^tA  {x})
\end{align*}
where \(\mu_{k_1, k_2}^{b, d, e}\) is the measure defined by
\begin{align*}
\int_{0}^{2\pi} f\left(2^{k_1} \cos \theta, \, 2^{k_2} \sin \theta, \, e \cdot 2^{k_1 + k_2} \sin \theta \cos \theta + b \cdot 2^{2k_1} \cos^2 \theta + d \cdot 2^{2k_2} \sin^2 \theta \right) d\theta.
\end{align*}
By expressing the third coordinate as \((b \cdot 2^{2k_1} - d \cdot 2^{2k_2}) \cos 2\theta + e \cdot 2^{k_1 + k_2} \sin 2\theta + (b \cdot 2^{2k_1} + d \cdot 2^{2k_2})\) through dilation and restricting the integration to \(\left[\frac{\pi}{4}, \frac{3\pi}{4}\right]\) as outlined in Section~\ref{3strpf}, we redefine the measure \(\mu_{k_1, k_2}^{b, d, e}\) as
\begin{align}\label{redm}
 \int_{\frac{\pi}{4}}^{\frac{3\pi}{4}} f\left(2^{k_1} \cos \theta, \, 2^{k_2} \sin \theta, \, (b \cdot 2^{2k_1} - d \cdot 2^{2k_2}) \cos 2\theta + e \cdot 2^{k_1 + k_2} \sin 2\theta + (b \cdot 2^{2k_1} + d \cdot 2^{2k_2}) \right) d\theta.
\end{align}
Then the main $L^2$ estimate is to be (\ref{dcy2}) where
the corresponding operator $\mathcal{C}^s_K \phi(x)$  is now given by
\begin{align*}
 \mathcal{C}^{s}_{A,K} \phi(x)  &:= \int e^{-2\pi i \lambda 2^{k_1+k_2} \Phi^A(x,y)} \beta(2^s(x-y))  \phi(y)\, dy
 \end{align*}
 where the phase function $\Phi^A(x,y)$ is given by
\begin{align}\label{k4}
(e(x-y)+(x+y))\sqrt{1-(x-y)^2} + b\frac{2^{2k_1}(x-y)^2}{2^{k_1+k_2}}+d\frac{2^{2k_2} (1-(x-y)^2)}{2^{k_1+k_2}}. 
\end{align}
In order to prove the $L^p(\mathbb{H}^1)$ boundedness of $f\mapsto \sup_{K} |\mathcal{C}^{s}_{A,K}* f|$, we only need to show the decay estimate given by \eqref{dcy1}, which states that
\begin{align}\label{dcy22}
    \|\mathcal{C}^{s}_{A,K} \mathcal{P}_{k_1+k_2-\ell_2}^{2} \mathcal{P}_{\ell_1}^1\|_{L^2(\mathbb{R}) \mapsto L^2(\mathbb{R})} \le C 2^{-c (s+\ell_1+\ell_2)}.
\end{align}
since the other arguments for extending it to general $1<p<\infty$ are similar to those in the previous sections.
In order to prove (\ref{dcy22}), we first claim that
\begin{proposition}\label{prop45}
Given $s>0$ and $m\in\mathbb{Z}$, consider the operator $\mathcal{C}_{A,K}^{s,m}$ defined by $$	\mathcal{C}_{A,K}^{s,m}\phi(x):= \int e^{-2\pi i{\lambda}2^{k_1+k_2}{\Phi^A}(x,y)} \varphi(2^{-m} y)\beta(2^{s}(x-y)){\phi}(y)dy.      $$
Then, there exists $C_A$ only depending on matrix $A$ satisfying
\begin{align}\label{kk}
&\| \mathcal{C}^{s,m}_{A,K} \|_{L^2(\mathbb{R}) \mapsto L^2(\mathbb{R})}
 \leq C_{A}2^{-c_1s} \min\{ |\lambda 2^{k_1+k_2} 2^m|^{-c}, |\lambda 2^{k_1+k_2} |^{-c},|\lambda b 2^{2k_1}|^{-c},|\lambda  d2^{2k_2}|^{-c}\}.
\end{align}
This yields the decay $2^{-c(s+\ell_2)}$ in (\ref{dcy22}).
\end{proposition}

\begin{proof}[Proof of \eqref{kk}]
We consider the integral kernel of $ \mathcal{C}^{s,m}_{A,K} [  \mathcal{C}^{s,m}_{A,K} ]^*$, following a similar approach as in \eqref{nee}. Ignoring oscillatory effects for the moment, we deduce from the support condition of $\beta$ that $\|C_{A,K}^{s,m}\|_{L^2(\mathbb{R}) \to L^2(\mathbb{R})} \leq 2^{-s}.
$ To utilize the oscillatory integral, we compute the Hessian of the phase $\Phi_A(x, y)$:
\begin{align*}
[\Phi^A]_{xy}(x, y) &= \frac{(x + y) + e\left[ 3(x - y) - 2(x - y)^3 \right]}{\left( 1 - (x - y)^2 \right)^{3/2}} - b \frac{2^{2k_1 + 1}}{2^{k_1+k_2}} + d \frac{2^{2k_2 + 1}}{2^{k_1+k_2}}.
\end{align*}
We observe the following:
\begin{enumerate}
    \item There exists $d \in \mathbb{N}$ such that $
    \sum_{\alpha = 1}^d \left| \partial_y^\alpha \left[\Phi^A\right]_{xy}(x, y) \right| \geq c > 0,$    
    \item If $2^{m} \geq 2^{20} (|3e| + 1)$, then $ \left| \partial_y^{2} \left[\Phi^A\right]_{xy}(x, y) \right| \geq c \, 2^{m},$
     \end{enumerate}
where $c$ is independent of $k_1$ and $k_2$. 
The finite type condition of \([\Phi^A]_{xy}(x, y)\) is ensured because both the numerator and denominator are polynomials. To verify (2), we compute
\begin{align*}
\partial_y^{2} \left[\Phi^A\right]_{xy}(x, y) &= \frac{6(y - x)}{\left( 1 - (y - x)^2 \right)^{5/2}} + \frac{15 (y - x)^2 \left( -2e(x - y)^3 + 3e(x - y) + x + y \right)}{\left( 1 - (y - x)^2 \right)^{7/2}} \\
&\quad + \frac{3 \left( -2e(x - y)^3 + 3e(x - y) + x + y \right)}{\left( 1 - (y - x)^2 \right)^{5/2}}.
\end{align*}
Under the assumptions $2^m \geq 2^{20} (|3e| + 1)$ and $|x - y| \approx 2^{-s} \ll 1$, we find that
\[
\left| \partial_y^2 [\Phi^A]_{xy}(x, y) \right| \approx 2^m.
\] 
Returning to the proof, we now apply Lemma \ref{lpt} combined with observations $(1)$ and $(2)$  to obtain the bounds $
(\lambda 2^{k_1 + k_2} |x - z|)^{-1/d} \quad \text{and} \quad (\lambda 2^{k_1 + k_2} 2^m |x - z|)^{-1/3}$
on the right-hand side of \eqref{nee}. Consequently, we have 
$$\| \mathcal{C}^{s,m}_{A,K} \|_{L^2(\mathbb{R}) \mapsto L^2(\mathbb{R})} \le C_A \cdot 2^{-c_1s} \min\{ |\lambda 2^{k_1+k_2} 2^m|^{-c}, |\lambda 2^{k_1+k_2}|^{-c}\}.$$
Thee estimate directly gives \eqref{kk} under the assumption that  
\begin{align}\label{ass1}
	2^{k_1 + k_2} 2^m + (|3a| + 1) 2^{k_1 + k_2-s} \geq 2^{-5} \left(|b 2^{2k_1}| + |d 2^{2k_2}|\right),
\end{align}  
We now analyze the remaining situation.
In this remaining case, the contributions from $- b \frac{2^{2k_1 + 1}}{2^{k_1+k_2}} + d \frac{2^{2k_2 + 1}}{2^{k_1+k_2}}$ in the Hessian become dominant. Specifically, consider the case where \( |\lambda b 2^{2k_1}| \geq |\lambda d 2^{2k_2 + 3}| \) or \( |\lambda d 2^{2k_2}| \geq |\lambda b 2^{2k_1 + 3}| \). It follows that \( \left| [\Phi^A]_{xy}(x, y) \right| \approx \left| b \frac{2^{2k_1 + 1}}{2^{k_1+k_2}} \right| \) or \( \left| d \frac{2^{2k_2 + 1}}{2^{k_1+k_2}} \right| \). Applying Lemma \ref{lpt} with \( k = 1 \), this yields the bound \( (\lambda b 2^{2k_1} |x - z|)^{-1/2} \), which confirms \eqref{kk} for this case. Next, consider the situation where \( |\lambda d 2^{2k_2}| \approx |\lambda b 2^{2k_1}| \). This implies \( |\lambda 2^{k_1 + k_2}| \approx |\lambda b 2^{2k_1}| \approx |\lambda d 2^{2k_2}| \), which directly leads to \eqref{kk}.
\end{proof}
\subsection*{Proof of \eqref{dcy22}}By applying Proposition \ref{prop45}, we obtained the decay $2^{-c(s+\ell_2)}$ in (\ref{dcy22}). So, in order to obtain $2^{-\ell_1}$, it suffices to show
\begin{align}\label{dcc2}
\|\mathcal{C}^{s}_{A,K} \mathcal{P}_{k_1+k_2-\ell_2}^{2} \mathcal{P}^1_{\ell_1}\|_{L^2(\mathbb{R}) \mapsto L^2(\mathbb{R})} \le C_A2^{-\ell_1/2},\ 
\end{align}
under the condition that
\begin{align}\label{k3}
 2^{\ell_1}\ge 2^{100}\left( 2^{\ell_2}+(1+|e|)|\lambda 2^{k_1+k_2}|+|\lambda b 2^{2k_1}|+|\lambda  d2^{2k_2}|\right).
 \end{align}
 This yields $2^{-c\ell_2}$ in (\ref{dcy22}).

\begin{proof}[Proof of (\ref{dcc2})]

It suffices to prove (\ref{ogp}) and (\ref{obp}) for the phase function $\Phi^A(x,y)$ in (\ref{k4}).  The inequality (\ref{ogp}) follows in the similar manner. For the phase  $\Phi^A(x,y)$ defined in  (\ref{k4}), by using  $  |\xi|$ dominating condition of (\ref{k3}) with   $|x-y|\approx 2^{-s}\le 1/2$, one can obatin (\ref{k42}) as
\begin{align*} 
	&|\partial_y(\lambda 2^{k_1+k_2}\Phi^A(x,y)-\xi y)|\\
	&\ge |\xi|-   \left(|\lambda2^{k_1+k_2} 2^m|+ (1+|e|)|\lambda 2^{k_1+k_2}|+   |\lambda b 2^{2k_1}2^{-s}|+|\lambda  d2^{2k_2}2^{-s}| \right)\\
& \approx |\xi| \approx 2^{\ell_1}.
\end{align*}
which leads (\ref{obp}). Therefore, we obtain \eqref{dcc2}.
\end{proof}
 \subsection{ Symmetric Case $A_w=0$}
In this section, 
we consider
 \begin{align}\label{44m}
 \text{ if $A\neq \begin{pmatrix} c & 0 \\ 0 & c\cdot 2^{2a} \end{pmatrix}$  for $c\in \mathbb{R}\setminus \{0\}$, then }\|\mathcal{E}^2_Af\|_{L^p(\mathbb{R}^3)}\le C_{p, A}\|f\|_{L^p(\mathbb{R}^3)}. 
 \end{align}
 and
 \begin{align}\label{44m2}
\text{ if $A\neq \begin{pmatrix} c & 0 \\ 0 & c \end{pmatrix}$  for $c\in \mathbb{R}\setminus \{0\}$, then } \|\mathcal{E}^1_Af\|_{L^p(\mathbb{R}^3)}&\le C_{p, A}\|f\|_{L^p(\mathbb{R}^3)}.
\end{align}
In view of \eqref{pp2} and \eqref{redm}, it is enough to consider $ \mu_{k_1, k_2}^{b, d, e}\ast f$ 
where \(\mu_{k_1, k_2}^{b, d, e}\) is the measure defined by
\begin{align*}
 \int_{\frac{\pi}{4}}^{\frac{3\pi}{4}} f\left(2^{k_1} \cos \theta, \, 2^{k_2} \sin \theta, \, (b \cdot 2^{2k_1} - d \cdot 2^{2k_2}) \cos 2\theta + e \cdot 2^{k_1 + k_2} \sin 2\theta + (b \cdot 2^{2k_1} + d \cdot 2^{2k_2}) \right) d\theta.
\end{align*}
The Euclidean Fourier transform of the measure $\mathcal{F}^{1,2,3}\mu_{k_1,k_2}^{b,d,e}$ is 
\begin{align}\label{042d}
e^{2\pi i(b 2^{2k_1}+ d2^{2k_2})\xi_3}\int_{[\frac{\pi}{4}, \frac{3\pi}{4}]}e^{-2\pi i \phi(\theta,(2^{k_1}\xi_1,2^{k_2}\xi_2, e2^{k_1+k_2}\xi_3, (b 2^{2k_1}- d2^{2k_2})\xi_3)}d\theta, 
\end{align} 
for $\phi(\theta,\eta_1,\eta_2,\eta_3,\eta_4)= (\eta_1,\eta_2,\eta_3,\eta_4)\cdot  (\cos\theta,\sin\theta,\cos 2\theta, \sin 2\theta)$. We shall prove
\begin{align}\label{12kk}
	\left|\int_{[\frac{\pi}{4}, \frac{3\pi}{4}]}e^{-2\pi i [ (\eta_1,\eta_2,\eta_3,\eta_4)\cdot (\cos\theta,\sin\theta,\cos 2\theta, \sin 2\theta) ]}d\theta  \right|\lesssim  (1+|\eta|)^{-1/4}.
\end{align} 
\begin{proof}[Proof of (\ref{12kk})]
Let $\phi(\theta,\eta)= (\eta_1,\eta_2,\eta_3,\eta_4)\cdot  (\cos\theta,\sin\theta,\cos 2\theta, \sin 2\theta)$. Let $e(\theta)= (\cos\theta,\sin\theta)$; then $e'(\theta)=(-\sin\theta,\cos\theta)$, which we denote by $e^{\perp}(\theta)$. Then from the pair of the first and third derivatives and the pair of the second and fourth derivatives below
\begin{align*}
\partial_{\theta} \phi(\theta,\eta)&= (\eta_1,\eta_2)\cdot e^{\perp}(\theta)+2(\eta_3,\eta_4) \cdot e^{\perp}(2\theta)\\
\partial^2_{\theta} \phi(\theta,\eta)&= - (\eta_1,\eta_2)\cdot e(\theta)-4(\eta_3,\eta_4)\cdot e(2\theta)\\
\partial^3_{\theta} \phi(\theta,\eta)&= -(\eta_1,\eta_2)\cdot e^{\perp}(\theta)-8(\eta_3,\eta_4) \cdot e^{\perp}(2\theta)\\
\partial^4_{\theta} \phi(\theta,\eta)&=  (\eta_1,\eta_2)\cdot e(\theta)+16(\eta_3,\eta_4)\cdot e(2\theta),
\end{align*}
one can observe that  $$ \sum_{k=1}^4|\partial^{k}_{\theta} \phi(\theta,\eta)|\ge 2^{-100}( |(\eta_1,\eta_2)|+|(\eta_3,\eta_4)|).$$
From this observation with van der Corput lemma in page 334 of \cite{MR1232192}, we can get (\ref{12kk}).
\end{proof}

In \cite{AIF_1992__42_3_637_0}, Theorem $3.2$ proved that lacunary maximal operators are \(L^p\)-bounded provided the Fourier transform of the associated measure satisfies a decay condition. But, the coefficient $(b 2^{2k_1-1}- d2^{2k_2-1})$ is needed to be handle carefully. To establish \eqref{44m}, we prove
\begin{align}\label{1para}
\left\|\sup_{k_1,k_2} \left| \mu_{k_1,k_2}^{b,d,e}\ast f \right|\right\|_{L^p(\mathbb{R})} \le C_{p, A} \|f\|_{L^p(\mathbb{R})},	
\end{align}
where the measure \(\mu_{k_1,k_2}^{b,d,e}\) is defined by
\[
 \int_{\frac{\pi}{4}}^{\frac{3\pi}{4}} f\left( 2^{k_1}\cos \theta, 2^{k_2}\sin \theta ,(b \cdot 2^{2k_1} - d \cdot 2^{2k_2})\cos 2\theta + e \cdot 2^{k_1 + k_2} \sin 2\theta + (b \cdot 2^{2k_1} + d \cdot 2^{2k_2}) \right) \, d\theta.
\]
Here, we outline the proof for \eqref{1para}, a different part not covered in \cite{AIF_1992__42_3_637_0}.

\begin{proof}[Proof of \eqref{1para}]
Decompose
\begin{align*}
\mathcal{F}^{1,2,3}\mu_{k_1,k_2}^{b,d,e}&=\left\{\psi(b2^{2k_1}\xi)+\psi(d2^{2k_2}\xi)\psi^c(b2^{2k_1}\xi)+\sum_{\ell_1,\ell_2=0}^{\infty}\varphi\left(\frac{d2^{2k_2}\xi}{2^{\ell_2}}\right)\varphi\left(\frac{b2^{2k_1}\xi}{2^{\ell_1}}\right)\right\}\mathcal{F}^{1,2,3}\mu_{k_1,k_2}^{b,d,e}.
\end{align*}
As we did in \eqref{lmdo}, the first two parts can be controlled by the composition of Hardy--Littlewood maximal operator and the maximal operator $\sup_{k_1} | \mu_{k_1,k_2}^{b,0,e}\ast f(x)|$ or $\sup_{k_2} | \mu_{k_1,k_2}^{0,d,e}\ast f(x)|$. To treat the two maximal operators, according to the Lifting Lemma (see page $484$ of \cite{MR1232192}), it suffices to show the $L^p$ boundedness of  $\sup_{k_1,k_2} | \tilde{\mu}_{k_1,k_2}^{0,d,e}\ast f |$ and $\sup_{k_1,k_2} | \tilde{\mu}_{k_1,k_2}^{b,0,e}\ast f |$. Here, the measure $\tilde{\mu}_{k_1,k_2}^{b,d,e}$ is defined by
\begin{align*}
	 \int_{\frac{\pi}{4}}^{\frac{3\pi}{4}} f\left( 2^{k_1}\cos \theta, 2^{k_2}\sin \theta ,(b \cdot 2^{2k_1} - d \cdot 2^{2k_2})\cos 2\theta+ (b \cdot 2^{2k_1} + d \cdot 2^{2k_2}), e \cdot 2^{k_1 + k_2} \sin 2\theta \right) \, d\theta.
\end{align*}
Let us say $b=0$. By Theorem $3.2$ in \cite{AIF_1992__42_3_637_0}, the Fourier decay condition \eqref{12kk} ensures the $L^p$ boundedness of these maximal operators.\\
To handle the last part $\varphi\left(\frac{d2^{2k_2}\xi_3}{2^{\ell_2}}\right)\varphi\left(\frac{b2^{2k_1}\xi_3}{2^{\ell_1}}\right)\mathcal{F}^{1,2,3}\mu_{k_1,k_2}^{b,d,e}$, we denote that 
\begin{align*}
	\mathcal{F}(\mathcal{P}^1_{\ell_1-2k_1}f)(\cdot, \xi_3):=\mathcal{F}f(\cdot, \xi_3)\varphi\left(\frac{b2^{2k_1}\xi_3}{2^{\ell_1}}\right),\\
\mathcal{F}(\mathcal{P}^2_{\ell_2-2k_2}f)(\cdot, \xi_3):=\mathcal{F}f(\cdot, \xi_3)\varphi\left(\frac{d2^{2k_2}\xi_3}{2^{\ell_2}}\right).
\end{align*}
Under the assumption in \eqref{44m}, we analyze two distinct cases based on the \( b \), \( d \), and \( e \).\\
When \( b/d \notin 2^{2\mathbb{Z}} \) and \( e \neq 0 \), we apply the Lifting Lemma. Then the decay estimate  \eqref{12kk} gives the following.
\begin{align*}
	|\mathcal{F}\tilde{\mu}_{k_1,k_2}^{b,d,e}(\xi_1,\xi_2,\xi_3,\xi_4)| &\lesssim |(b \cdot 2^{2k_1} - d \cdot 2^{2k_2})\xi_3|^{-1/4} + |e \cdot 2^{k_1 + k_2} \xi_4|^{-1/4}
\end{align*}
In the case where either \( b/d \notin 2^{2\mathbb{Z}} \) or \( e \neq 0 \) not both, we utilize the following estimate:
\begin{align}\label{dcyest}
	|\mathcal{F}\mu_{k_1,k_2}^{b,d,e}(\xi_1,\xi_2,\xi_3)| &\lesssim \min\left\{ |(b \cdot 2^{2k_1} - d \cdot 2^{2k_2})\xi_3|^{-1/4}, \, |e \cdot 2^{k_1 + k_2} \xi_3|^{-1/4} \right\}
\end{align}
Under the support $\varphi\left(\frac{d2^{2k_2}\xi_3}{2^{\ell_2}}\right)\varphi\left(\frac{b2^{2k_1}\xi_3}{2^{\ell_1}}\right)$, one can see that 
$|(b \cdot 2^{2k_1} - d \cdot 2^{2k_2})\xi|+|e2^{k_1+k_2}\xi|\gtrsim_{b,d}\max\{2^{\ell_1},2^{\ell_2}\} $. Combining these estimates and applying the square function method, it is easy to verify that
\begin{align}\label{l22}
	\|\sup_{k_1,k_2} | \mu_{k_1,k_2}^{b,d,e}\ast \mathcal{P}^1_{\ell_1-k_1}\mathcal{P}^2_{\ell_2-k_2}f|\|_{L^2}\lesssim 2^{-\varepsilon(\ell_1+\ell_2)} \|f\|_{L^2}.
\end{align}
Moreover, by utilizing the shifted maximal operators defined on page 741 of \cite{MR3231215} or page 18 of \cite{MR1232192}, we obtain
\begin{align}\label{lpp}
	\|\sup_{k_1,k_2} | \mu_{k_1,k_2}^{b,d,e}\ast \mathcal{P}^1_{\ell_1-k_1}\mathcal{P}^2_{\ell_2-k_2}f|\|_{L^p}\lesssim |\ell_1\ell_2|^{1/p} \|f\|_{L^p}.
\end{align}
For more details on the arguments involving shifted maximal operators, refer to the Section $4.2$ in \cite{MR4701897}. By the usual interpolation argument using \eqref{l22} and \eqref{lpp}, we can handle the summation over \(\ell_1\) and \(\ell_2\). Consequently, we obtain the \(L^p\)-boundedness of \(\sup_{k_1,k_2} | \mu_{k_1,k_2}^{b,d,e} \ast f(x) |\).
\end{proof}
Finally, we consider the one parameter case. Following a similar approach as above and in view of \eqref{dcyest}, we obtain
\[
|\mathcal{F}\gamma^{b,d}_{k,k}(\xi)| \lesssim \left( |(b - d) \cdot 2^{2k} \xi| + |e \cdot 2^{2k} \xi| \right)^{-1/4}
\]
by \eqref{12kk}. Under the assumption that \(b \neq d\) or \(e \neq 0\), this yields \eqref{1para} for \(k_1 = k_2\), which in turn implies \eqref{44m2}.

\end{document}